\documentclass[letterpaper, 10 pt, conference]{ieeeconf}  

\IEEEoverridecommandlockouts                   

\usepackage{hyperref}
\usepackage{graphics}
\usepackage{epsfig} 
\usepackage{amsmath}
\usepackage{amssymb}
\usepackage{dsfont}
\usepackage{enumerate}
\usepackage[tight]{subfigure}
\usepackage{multirow}
\usepackage{bm}
\usepackage{tikz}

\usepackage[noadjust]{cite}

\usetikzlibrary{arrows.meta}
\usetikzlibrary{positioning}

\usepackage{color}

\title{
Nonlinear network identifiability: The static case
}

\author{Renato Vizuete and Julien M. Hendrickx
\thanks{*This work was supported by F.R.S.-FNRS via the \emph{KORNET} project and via the Incentive Grant for Scientific Research (MIS) \emph{Learning from Pairwise Comparisons}, and by the \emph{RevealFlight} Concerted Research Action (ARC) of the Fédération Wallonie-Bruxelles.}
\thanks{R.~Vizuete and J.~M.~Hendrickx are with ICTEAM institute, UCLouvain, B-1348, Louvain-la-Neuve, Belgium. R.~Vizuete is a FNRS Postdoctoral Researcher - CR.
{\tt\small renato.vizueteharo@uclouvain.be},
{\tt\small julien.hendrickx@uclouvain.be}\protect.}
}

\newcommand{\vertiii}[1]{{\left\vert\kern-0.25ex\left\vert\kern-0.25ex\left\vert #1 
    \right\vert\kern-0.25ex\right\vert\kern-0.25ex\right\vert}}

\newtheorem{definition}{Definition}
\newtheorem{assumption}{Assumption}
\newtheorem{theorem}{Theorem}
\newtheorem{corollary}{Corollary}
\newtheorem{proposition}{Proposition}
\newtheorem{lemma}{Lemma}
\newtheorem{remark}{Remark}
\newtheorem{example}{Example}

\newcommand{\F}{\mathcal{F}}

\newcommand{\Fone}{\mathcal{F}_Z}

\newcommand{\Ftwo}{\mathcal{F}_{Z,NL}}

\newcommand{\R}{\mathbb R}

\begin{document}

\maketitle
\thispagestyle{empty}

\begin{abstract}
We analyze the problem of network identifiability with nonlinear functions associated with the edges. We consider a static model for the output of each node and by assuming a perfect identification of the function associated with the measurement of a node, we provide conditions for the identifiability of the edges in a specific class of functions. First, we analyze the identifiability conditions in the class of all nonlinear functions and show that even for a path graph, it is necessary to measure all the nodes except by the source. Then, we consider analytic functions satisfying $f(0)=0$ and we provide conditions for the identifiability of paths and trees. Finally, by restricting the problem to a smaller class of functions where none of the functions is linear, we derive conditions for the identifiability of directed acyclic graphs. 
Some examples are presented to illustrate the results.
\end{abstract}

\section{Introduction}

Networked systems composed by nodes or subsystems interacting with each other are ubiquitous \cite{bullo2022lectures}. In several of these systems, the knowledge of the dynamics associated with edges is essential for the analysis of the system and design of control algorithms. However, the identification of the networked systems from partial measurements without disconnections of some parts of the network can be really challenging since a measured signal depends on the combination of the dynamics of potentially many edges.

There has been some recent works in the linear case on the conditions of identifiability: when is it possible to unambiguously recover local dynamics from a set of measured nodes? This question is important in order to design experiments and position sensors and excitations \cite{ramaswamy2019generalized,bombois2023informativity,kivits2023identifiability}. This depends mainly on the topology of the network and on the position of the excitation and measured signal. Graph theoretical conditions are available in full measurement case or full excitation \cite{hendrickx2019identifiability}, but not in the general case yet \cite{legat2020local,legat2021path,bazanella2019network,cheng2023necessary}. 
However, most actual systems of interest are nonlinear, including many different research fields like coupled oscillators \cite{dorfler2014synchronization}, gene regulatory networks \cite{pan2012reconstruction}, biochemical reaction networks \cite{aalto2020gene}, social networks \cite{bizyaeva2023nonlinear}, among others. While linear systems usually provide a local approximation of nonlinear phenomena, no one to the best of our knowledge has studied the identifiability question for nonlinear systems.

The identification of a nonlinear system is itself a challenging problem due to the variety of potential models (e.g., Hammerstein, Wiener, Volterra series) and the constant arising of new formulations for particular applications. Depending on the type of nonlinearities and its location (i.e., at the level of inputs, outputs or in the middle of interactions), certain models can be more suited for specific applications, while others could not give a good description of some systems \cite{janczak2004identification,nelles2020nonlinear,paduart2010identification}. In addition to the complexity of a single nonlinear model, a network involves several nonlinear systems associated with the edges, which generates complex collective behaviors and increase considerably the difficulties of the identification problem. 

In the nonlinear case, the conditions for identifiability of networks do not depend only on the network topology but also on the types of nonlinear functions. 
For instance, trigonometric functions in coupled oscillators  \cite{dorfler2014synchronization} are very different from the activation functions in neural networks that can be nondifferentiable \cite{aggarwal2018neural}. Furthermore, in heterogeneous networks, different types of functions could be associated with edges in the same network. In addition,
the class of functions considered for the problem of identifiability could be determinant. It is clear that if we restrict the problem to a small class of functions, the conditions for identifiability of a network could be relaxed, but the functions in the class could not fit real models. Moreover, properties of functions such as continuity, differentiability, analyticity, etc., could play an important role in the determination of conditions for the identifiability of networks.

We study here the question of identifiability in the nonlinear setting, assuming in this first work that the local dynamics have very simple structure (i.e., the output of a node is entirely determined by static interactions with the neighbors). We show that, surprisingly, the conditions for identifiability in directed acyclic graphs are weaker  than in the linear case, provided that the dynamics are indeed not linear, and do not involve constant output component (when they do, the problem is indeed unsolvable).
We explain this by the fact that in the linear case, the loss of identifiability often results from ambiguities made possible by the superposition principle/superposition of signal, which is no longer possible in the nonlinear case.

In this work, we provide a formulation of the network identifiability problem in the nonlinear case, by considering a static model in the edges. By restricting the problem to a specific class of functions, we provide identifiability conditions for paths and trees. Furthermore, by considering a smaller class of functions, we derive conditions for the identifiability of directed acyclic graphs.

\section{Problem formulation}

\subsection{Model class}

Since the type of nonlinear dynamics in a network can be really complex, in this preliminary work, we will consider a static additive model to focus on the effect of the nonlinearities. Therefore, we exclude dynamical processes that involve any memory at the level of nodes or edges. Our objective is to generalize the results of this paper to more complex dynamical models in future works. 

For a network composed by $n$ nodes, we consider that the output of each node $i$ is given by:
\begin{equation}\label{eq:nonlinear_model}
y_i^k=\sum_{j\in \mathcal{N}_i}f_{i,j}(y_j^{k-1})+u_i^{k-1}, \;\; \text{for all  } i\in\{1,\ldots,n\},  
\end{equation}
where the superscripts denote the value of the inputs and outputs at the specific time instant, $f_{i,j}$ is a nonlinear function, $\mathcal{N}_i$ is the set of in-neighbors of node $i$, and $u_i$ is an external excitation signal. The node $i$ is not included in $\mathcal{N}_i$, since it would imply a dynamical process at the level of the node. The model  \eqref{eq:nonlinear_model} corresponds to a nonlinear static version of the model considered in \cite{hendrickx2019identifiability,legat2020local,legat2021path}, where the nonlinearities are located in the edges. In this case, the output of a node $i$ is determined by its own excitation signal $u_i$, and the outputs of the neighbors $y_j$ affected by a nonlinear function $f_{i,j}$ associated with the edge that connects the neighbor.
Notice that when the functions $f_{i,j}$ in \eqref{eq:nonlinear_model} are linear, the conditions for identifiability of linear networks derived in \cite{hendrickx2019identifiability,legat2020local,legat2021path} also hold. In this work, we will consider analytic functions with a Taylor series that converges to the function for all $x\in\R$. This representation as power series will allow us to derive conditions for the identifiability of nonlinear networks.

Model \eqref{eq:nonlinear_model} corresponds to the full excitation case where all the nodes are excited. The nonzero functions $f_{i,j}$ between the agents define the topology of the network $G$, forming the set of edges $E$. In this work, we do not consider multi-edges between two nodes.

\begin{assumption}\label{ass:full_excitation}
    The topology of the network is known, where the presence of an edge implies a nonzero function.
\end{assumption}

Assumption~\ref{ass:full_excitation} implies that we know which nodes are connected by nonzero functions. 
The objective is to determine which nodes need to be measured to identify all the nonlinear functions in the network.

Similarly to \cite{hendrickx2019identifiability,legat2020local,legat2021path}, for the identification process we assume that in an ideal scenario the relations between excitations and outputs of the nodes have been perfectly identified. In this work, we restrict our attention to networks that do not contain any cycle (i.e., directed acyclic graphs). 
This implies that when we measure a node $i$, we identify the function $F_i^k$:
\begin{multline}\label{eq:function_Fi}
    \!\!\!\!\!y_i^k\!=\!u_i^{k-1}\!+F_i^k(u_1^{k-2},\ldots,u_1^{k-m_1},\ldots,u_{n_i}^{k-2},\ldots,u_{n_i}^{k-m_{n_i}}),\\
    1,\dots,n_i\in \mathcal{N}_i^p,
\end{multline}
where $\mathcal{N}_i^p$ denotes the set of nodes that have a path to the measured node $i$.
The function $F_i^k$ is implicitly defined by \eqref{eq:nonlinear_model} and only depends on a finite number of inputs due to the absence of memory on the edges and nodes, and the absence of cycles. 
With a slight abuse of notation, we use the superscript in the function $F_i^{k-s}$ to indicate that all the inputs in \eqref{eq:function_Fi} are delayed by $s$.   

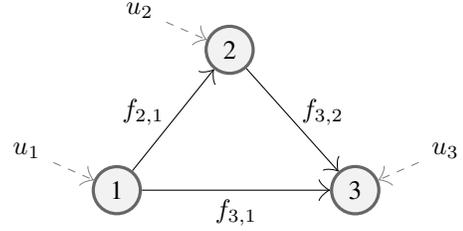
\begin{figure}
    \centering

\vspace{1.5mm}
    
    \begin{tikzpicture}
    [
roundnodes/.style={circle, draw=black!60, fill=black!5, very thick, minimum size=1mm},roundnode/.style={circle, draw=white!60, fill=white!5, very thick, minimum size=1mm}
]
\node[roundnodes](node1){1};
\node[roundnodes](node2)[above=of node1,yshift=2mm,xshift=1.5cm]{2};
\node[roundnodes](node3)[right=2.5cm of node1]{3};
\draw[-{Classical TikZ Rightarrow[length=1.5mm]}] (node1) to node [left,swap,yshift=1mm] {$f_{2,1}$} (node2);
\draw[-{Classical TikZ Rightarrow[length=1.5mm]}] (node2) to node [right,swap,yshift=1mm] {$f_{3,2}$} (node3);
\draw[-{Classical TikZ Rightarrow[length=1.5mm]}] (node1) to node [below,swap] {$f_{3,1}$} (node3);

\node[roundnode](u1)[above=of node1,yshift=-12mm,xshift=-12mm]{$u_1$};
\node[roundnode](u2)[above=of node2,yshift=-12mm,xshift=-12mm]{$u_2$};
\node[roundnode](u3)[above=of node3,yshift=-12mm,xshift=12mm]{$u_3$};

\draw[gray,dashed,-{Classical TikZ Rightarrow[length=1.5mm]}] (u1) -- (node1);
\draw[gray,dashed,-{Classical TikZ Rightarrow[length=1.5mm]}] (u2) -- (node2);
\draw[gray,dashed,-{Classical TikZ Rightarrow[length=1.5mm]}] (u3) -- (node3);

\end{tikzpicture}
    \caption{The function $F_3^k$ associated with the measurement of the node 3 depends on the past inputs of the nodes 1 and 2 that have a path to the node 3. 
    }
    \label{fig:DAG_Fone}
\end{figure}

\begin{example}
    Let us consider the graph in Fig.~\ref{fig:DAG_Fone} where the measurement of the node 3 provides the output:
    \begin{align}
        y_3^k&=u_3^{k-1}+F_3^k\nonumber\\
        &=u_3^{k-1}+f_{3,2}(u_2^{k-2}+f_{2,1}(u_1^{k-3}))+f_{3,1}(u_1^{k-2}).\label{eq:DAG_different_paths}
    \end{align}
    We can observe that the function $F_3^k$ depends on the inputs of the nodes 1 and 2 that have a path to the node 3.
\end{example}

\subsection{Identifiability}

The identifiability problem is related to the possibility of identifying the functions $f_{i,j}$ based on several measurements. For this, we introduce the following relationship between the measurements and the functions $f_{i,j}$.

\begin{definition}[Set of measured functions]
    Given a set of measured nodes $\mathcal{N}^m$, the set of measured functions $F(\mathcal{N}^m)$ associated with $\mathcal{N}^m$ is given by:
    $$
    F(\mathcal{N}^m):=\{F_i^k\;|\;i\in \mathcal{N}^m\}.
    $$
\end{definition}
We say that a function $f_{i,j}$ associated with an edge satisfies $F(\mathcal{N}^m)$ if $f_{i,j}$ can lead to $F(\mathcal{N}^m)$ through \eqref{eq:nonlinear_model}.

For completely arbitrary functions, the identifiability problem can be really challenging or even unrealistic. For this reason, we restrict the identifiability problem to a certain class of functions $\F$, which implies that the functions associated with the edges belong to $\F$ and that the identifiability is considered only among the functions belonging to $\F$. The different classes of functions will be specified depending on the results.

\begin{definition}[Edge identifiable]
In a network $G$, an edge $f_{i,j}$ is identifiable in a class $\F$ if given a set of measured functions $F(\mathcal{N}^m)$, every set of functions in $\F$ leading to $F(\mathcal{N}^m)$ has the same $f_{i,j}$.
\end{definition}

\begin{definition}[Network identifiable]\label{def:network_id}
A network $G$ is identifiable in a class $\F$ if all the edges are identifiable in the class $\F$.
\end{definition}

The function $F_i^k$ in \eqref{eq:function_Fi} is the most complete information that we can obtain when we measure a node $i$. This implies that if it is not possible to identify the functions $f_{i,j}$ with $F_i^k$, these edges are unidentifiable. On the contrary, if the functions $f_{i,j}$ are identifiable, it seems reasonable that under some conditions, the function $F_i^k$ can be well approximated after sufficiently long experiments, which could allow us to identify the functions $f_{i,j}$ approximately.

\subsection{First results}

We provide a result about the information that we can obtain with the measurement of sinks and sources\footnote{A source is a node with no incoming edges. A sink is a node with no outgoing edges.}.

\begin{proposition}[Sinks and sources]\label{prop:sinks_sources}
The measurement of the sources is never necessary for the identifiability of the network. The measurement of all the sinks is necessary for the identifiability of the network. 
\end{proposition}
\begin{proof}
    First, the measurement of any source $j$ generates the output
    $
    y_j^k=u_j^{k-1},
    $
    which does not provide any  information about functions associated with edges in the network.
    Next, let us consider a sink $i$ with $m$ incoming edges. The measurement of this sink provides an output:
    \begin{equation}\label{eq:output_sink_functions}
    y_i^k=u_i^{k-1}+f_{i,1}(y_1^{k-1})+\cdots+f_{i,m}(y_m^{k-1}),    
    \end{equation}
    and it is the only way of obtaining information of the functions $f_{i,1},\ldots,f_{i,m}$. Thus, the measurement of all the sinks is necessary. 
\end{proof}

The following lemma provides a result about the structure of the function $F_i^k$ associated with the measurement of a node $i$ with respect to the excitation signals of the in-neighbors.

\begin{lemma}\label{lemma:unique_functions}
    Let $j$ be an in-neighbor of a measured node $i$ and the function $F_i^k$. Then, by assuming all the variables but $u_j^{k-2}$ constant, we have:
    $$
    F_i^k=\alpha+f_{i,j}(u_j^{k-2}+\beta),
    $$
    where $\alpha$ and $\beta$ are constants with respect to $u_j^{k-2}$.
\end{lemma}

\begin{proof}
    According to \eqref{eq:function_Fi}, the function $F_i^k$ of a measured node $i$ is given by:
    \begin{align}
        F_i^k&=\sum_{\ell=1}^m f_{i,\ell}(y_\ell^{k-1})\nonumber\\
        &=\sum_{\ell=1}^m f_{i,\ell}(u_\ell^{k-2}+F_{\ell}^{k-1}) \label{eq:sum_measure_node},
    \end{align}
    where $m$ is the number of in-neighbors of the node $i$. All the functions $F_{\ell}^{k-1}$ depend on inputs delayed by 1, which implies that no $F_{\ell}^{k-1}$ depends on $u_j^{k-2}$.
Finally, no $f_{i,p}$ with $p\neq j$ can be a function of $u_j^{k-2}$ since there are no multi-edges. 
\end{proof}

Lemma~\ref{lemma:unique_functions} implies that $f_{i,j}$ in \eqref{eq:sum_measure_node} is the only function that depends on $u_j^{k-2}$, and $F_{j}^{k-1}$ does not depend on $u_j^{k-2}$.

\section{Paths and trees}

\subsection{Strong requirements for general nonlinear functions}

Since the conditions for the identifiability of linear networks are based on the existence of paths in the network that carry information from the excited nodes to the measured nodes \cite{hendrickx2019identifiability,legat2021path}, we first focus on the conditions for the identifiability of a path graph \footnote{A path graph is a graph that can be drawn so that all the nodes and edges lie on a single straight line.} in the nonlinear case.

In the linear case, for this graph topology we only need to measure the sink to identify all the transfer functions of the network thanks to the superposition principle \cite{hendrickx2019identifiability}. However, this is not true for the nonlinear case. 

\begin{example}[Path graph]
Fig.~\ref{fig:path_graph} presents a simple path graph with 3 nodes where the measurement of the sink is not enough to identify the network when general nonlinear functions are considered.     
\end{example}

\begin{proposition}[General nonlinear functions]\label{prop:path_general}
    For identifiability of a path graph in the class of general nonlinear functions, it is necessary to measure all the nodes except by the source.
\end{proposition}
\begin{proof}
 Let us consider a path graph with $n>2$ nodes and a node $i$ in the middle, which is neither the source nor the sink. The output of the node $i+1$ is given by:
\begin{align}
y_{i+1}^k&=u_{i+1}^{k-1}+F_{i+1}^k\label{eq:F_ik}\\
&=u_{i+1}^{k-1}+f_{i+1,i}(u_i^{k-2}+f_{i,i-1}(u_{i-1}^{k-3}+F_{i-1}^{k-2})). \nonumber
\end{align}
If the node $i$ is not measured and we consider the functions $\tilde f_{i,i-1}(x)=f_{i,i-1}(x)+\gamma$ and $\tilde f_{i+1,i}(x)=f_{i+1,i}(x-\gamma)$ with $\gamma\neq 0$, the function $F_{i+1}^k$ in \eqref{eq:F_ik} is the same, which implies that the path graph cannot be identified. 

On the other hand, if we measure all the nodes, we know the function $F_j^k$ associated with any node $j$ in the network. Let us consider a node $i$ with an in-neighbor and we set all the inputs to 0 except $u_i^{k-1}$. Then, the measurement of the node~$i$ gives us:
\begin{align}
    y_i^k&=u_i^{k-1}+F_i^k\label{eq:F_iik}\\
    &=u_i^{k-1}+f_{i,i-1}(u_i^{k-2}+F_{i-1}^{k-1}(0))\nonumber
\end{align}
If there is another function $\tilde f_{i,i-1}$ satisfying $F_{i}^k$ in \eqref{eq:F_iik}, we would have for all $u_i^{k-2}\in\R$:
$$
f_{i,i-1}(u_i^{k-2}+F_{i-1}^{k}(0))=\tilde f_{i,i-1}(u_i^{k-2}+F_{i-1}^{k}(0)),
$$
which implies that $f_{i,i-1}=\tilde f_{i,i-1}$ and we can identify $f_{i,i-1}$.
Following a similar approach for the other nodes, we can identify all the nonlinear functions in the path graph. Finally, by Proposition~\ref{prop:sinks_sources}, it is never necessary to measure the source.
\end{proof}

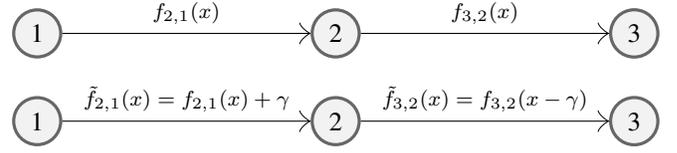
\begin{figure}
    \centering

\vspace{1mm}
    
    \begin{tikzpicture}
    [
roundnodes/.style={circle, draw=black!60, fill=black!5, very thick, minimum size=1mm},
roundnode/.style={circle, draw=black!60, fill=red!30, very thick, minimum size=1mm}
]

\node[roundnodes](node1){1};
\node[roundnodes](node2)[right=3.3cm of node1]{2};
\node[roundnodes](node3)[right=3.3cm of node2]{3};

\draw[-{Classical TikZ Rightarrow[length=1.5mm]}] (node2) to node [above,swap] {\footnotesize $f_{3,2}(x)$} (node3);
\draw[-{Classical TikZ Rightarrow[length=1.5mm]}] (node1) to node [above,swap] {\footnotesize $f_{2,1}(x)$} (node2);

\node[roundnodes](node4)[below=of node1, yshift=5mm]{1};
\node[roundnodes](node5)[right=3.3cm of node4]{2};
\node[roundnodes](node6)[right=3.3cm of node5]{3};

\draw[-{Classical TikZ Rightarrow[length=1.5mm]}] (node5) to node [above,swap] {\footnotesize $\tilde{f}_{3,2}(x)=f_{3,2}(x-\gamma)$} (node6);
\draw[-{Classical TikZ Rightarrow[length=1.5mm]}] (node4) to node [above,swap] {\footnotesize $\tilde{f}_{2,1}(x)=f_{2,1}(x)+\gamma$} (node5);

\end{tikzpicture}
\caption{A path graph with 3 nodes and different nonlinear  functions that satisfy $F_3^k=\tilde F_3^k$. For any $\gamma\neq 0$, the measurement of the sink is not enough for the identification of the network.}
    \label{fig:path_graph}
\end{figure}

Notice that in the proof of Proposition~\ref{prop:path_general} we do not use properties of analytic functions and the results are also valid for nonlinear functions that are not analytic.
Proposition~\ref{prop:path_general} shows that even in a simple graph topology like a path graph, which is the key for the identification of more complex network topologies, the identification of general nonlinear functions cannot be performed by only measuring the sink. This is due to the constant factor associated with the static behavior of the function. 

\subsection{Identifiability conditions for functions with no constant effect}

We restrict the identifiability problem to a smaller class of functions without a static component.

\begin{definition}[Class of functions $\Fone$]
Let $\Fone$ be the class of functions $f:\R\to\R$ with the following properties:
\begin{enumerate}
    \item $f$ is analytic in $\R$.
    \item $f(0)=0$.
\end{enumerate}
\end{definition}
Thus, we consider that for any function in $\Fone$, the Taylor series at 0 converges to the function for all $x\in \R$. Since the power series is unique, there is no loss of generality by considering the Taylor series at 0 and not centered at a different point in $\R$.

Notice that the class $\Fone$ encompasses numerous nonlinear functions \cite{abramowitz1965handbook}, including polynomial functions which are used for the approximation of continuous functions through the Weierstrass Approximation theorem \cite{llavona1986approximation}. Also, there is no loss of generality with the class $\Fone$ since the results are valid for the variable part of the functions.

\begin{lemma}\label{lemma:sinks}
    For identifiability in the class $\Fone$, the measurement of a node provides the identification of all the incoming edges of the node.
\end{lemma}
\begin{proof}
Let us consider a node $i$ with $m$ incoming edges. The output of the node $i$ is given by:
\begin{equation}\label{eq:F_incoming_edges}
y_i^k=u_i^{k-1}+F_i^k(u_1^{k-2},\ldots,u_m^{k-2},\ldots),  
\end{equation}
where $F_i^k$ is determined by the set of functions $\{ f \}$ associated with the edges of the network. Let us assume that there exists another set of functions $\{ \tilde f \}\neq \{ f \}$ such that:
$$
F_i^k(u_1^{k-2},\ldots,u_m^{k-2},\ldots)=\tilde F_i^k(u_1^{k-2},\ldots,u_m^{k-2},\ldots),
$$
where $\tilde F_i^k$ is composed by the functions in the set $\{ \tilde f\}$.
Let us choose a point $(u_j^{k-2},0,\ldots,0)$ with $j=1,\ldots,m$, such that all the inputs are set to zero except one of the inputs of the incoming edges of the node $i$. Then, since each function is in $\Fone$ and by Lemma~\ref{lemma:unique_functions} we have:
$$
F_i^k=f_{i,j}(u_j^{k-2})\quad \text{and} \quad \tilde F_i^k=\tilde f_{i,j}(u_j^{k-2}).
$$
Since we assume that $F_i^k=\tilde F_i^k$,
it yields:
$$
f_{i,j}(u_j^{k-2})=\tilde f_{i,j}(u_j^{k-2}), \text{ for all } u_j^{k-2}\in\R,
$$
which implies that $f_{i,j}=\tilde f_{i,j}$. Following a similar argument for each incoming edge of node $i$, we prove that all the functions associated with the incoming edges of the node $i$ are unique and can be identified.
\end{proof}

Notice that due to other possible paths from an in-neighbor $j$ of the node $i$, additional terms of the form $u_j^{k-r}, r>2$ could appear in \eqref{eq:F_incoming_edges}. However, they will always be delayed by virtue of Lemma~\ref{lemma:unique_functions}.

The following lemmas involve properties of analytic functions that will be used in the proof of the results in this section.

\begin{lemma}(Periodic functions)\label{lemma:periodic}
    If for some $p_0 \in \R$, an analytic function $f:\R \to \R$ is periodic for periods $p\in[p_0-\epsilon,p_0+\epsilon]$ with $\epsilon>0$, then the function $f$ is constant.
\end{lemma}
\begin{proof}
    The proof is left to Appendix~\ref{app:1}.
\end{proof}

\begin{lemma}\label{lemma:identification_one_function}
Given three non-zero analytic functions $f:\R\to\R$ and $g,\tilde g:\R^m\to \R$ satisfying $g(0)=\tilde{g}(0)=0$. If for all $x\in\R$, $y\in\R^m$, the functions $f$, $g$ and $\tilde g$ satisfy:
$$
f(x+g(y_1,\ldots,y_m))=f(x+\tilde{g}(y_1,\ldots,y_m)),
$$
then either $g=\tilde g$ or $f$ is constant.
\end{lemma}
\begin{proof}
    The proof is left to Appendix~\ref{app:2}.
\end{proof}

\begin{corollary}\label{corr:g_egal_g}
Under the same conditions as in Lemma~\ref{lemma:identification_one_function}, if $f(0)=0$, then
$$
g=\tilde{g}.
$$
\end{corollary}

\vspace{2mm}

\begin{proposition}[Paths]\label{prop:path_zero_function}
    For identifiability of a path graph in the class $\Fone$, it is necessary and sufficient to measure the sink.
\end{proposition}
\begin{proof}
Let us consider a path with $n$ nodes. The measurement of the sink gives us the output:
\begin{align}
    y_n^k&=u_n^{k-1}+F_n^k\nonumber\\
    &=u_n^{k-1}+f_{n,n-1}(u_{n-1}^{k-2}+F_{n-1}^{k-1})\label{eq:sink_path}. 
\end{align}
Let us assume that there is a set $\{ \tilde f\}\neq \{ f\}$ such that $F_n^k=\tilde F_n^k$, which  by \eqref{eq:sink_path} implies:
$$
f_{n,n-1}(u_{n-1}^{k-2}+F_{n-1}^{k-1})=\tilde f_{n,n-1}(u_{n-1}^{k-2}+\tilde F_{n-1}^{k-1}).
$$
By Lemma~\ref{lemma:sinks}, we can guarantee that $f_{n,n-1}=\tilde f_{n,n-1}$, and we have:
$$
f_{n,n-1}(u_{n-1}^{k-2}+F_{n-1}^{k-1})= f_{n,n-1}(u_{n-1}^{k-2}+\tilde F_{n-1}^{k-1}).
$$
Then, we use Corollary~\ref{corr:g_egal_g} to guarantee that $F_{n-1}^{k-1}=\tilde F_{n-1}^{k-1}$. Notice that now the identifiability problem is equivalent to having measured the node $n-1$ and by following a similar approach, we can continue with the identification of all the edges and guarantee that $\{ f \}=\{ \tilde f \}$, such that all the path can be identified.
\end{proof}

\begin{proposition}[Trees]\label{corr:trees}
For identifiability of a tree in the class $\Fone$, it is necessary and sufficient to measure all the sinks.
\end{proposition}
\begin{proof}
    From Proposition~\ref{prop:sinks_sources}, it is necessary to measure all the sinks. Let us consider an arbitrary tree and the measurement of a sink $i$. Let us assume that there are $m$ in-neighbors of the sink $i$ and there is a set $\{ \tilde f\}\neq \{ f \}$ such that $F_i^{k}=\tilde F_i^{k}$, which implies:
    \begin{equation}\label{eq:branches_tree}
        \sum_{\ell=1}^mf_{i,\ell}(u_\ell^{k-2}+F_\ell^{k-1})=\sum_{\ell=1}^m\tilde f_{i,\ell}(u_\ell^{k-2}+ \tilde F_\ell^{k-1}).
    \end{equation}
    Since in a tree, the functions $F_\ell^{k-1}$ do not have common inputs because they come from different branches, we can select a in-neighbor $j$ and set to zero the inputs of all the nodes that do not have a path to $j$, such that we have:
    $$
    f_{i,j}(u_j^{k-2}+F_j^{k-1})=\tilde f_{j,j}(u_j^{k-2}+\tilde F_j^{k-1}).
    $$
    Then, by using Lemma~\ref{lemma:sinks} and Corollary~\ref{corr:g_egal_g} we can guarantee that $f_{i,j}=\tilde f_{i,j}$ and $F_j^{k-1}=\tilde F_j^{k-1}$ for all $j=1,\ldots,m$, which is equivalent to having measured the in-neighbors of $i$. Then, we can continue with the identification of each branch independently and by following the same approach we can identify all the paths that finish in the sink $i$. Finally, by measuring the other sinks and following a similar approach, we can identify all the edges in the tree.
\end{proof}

\begin{remark}[Linear functions]
Notice that Propositions~\ref{prop:path_zero_function} and \ref{corr:trees} are also valid if all or some of the edges in the network contain pure linear functions. In the next section, we will provide stronger results in the identification of nonlinear networks when linear functions are excluded.   
\end{remark}

\section{Directed acyclic graphs}

Directed acyclic graphs encompass a large number of graph topologies that present specific characteristics that can be used for the derivation of conditions for identifiability~\cite{mapurunga2022excitation}.
Unlike a tree, in a directed acyclic graph, 
the functions $F_{\ell}^{k-1}$ in \eqref{eq:branches_tree} can have common variables due to several possible paths between two nodes with the same length, which makes impossible the application of Corollary~\ref{corr:g_egal_g}. In order to obtain a result similar to Corollary~\ref{corr:g_egal_g} that allow us to identify a directed acyclic graph, we consider a smaller class of functions.

\begin{definition}[Class of functions $\Ftwo$]
Let $\Ftwo$ be the class of functions $f:\R\to\R$ with the following properties:
\begin{enumerate}
    \item $f$ is analytic in $\R$.
    \item $f(0)=0$.
    \item The associated Taylor series $f(x)=\sum_{n=1}^\infty a_nx^n$ contains at least one coefficient $a_n\neq 0$ with $n>1$.    
\end{enumerate}
\end{definition}

The third property of the functions in $\Ftwo$ implies that none of the functions is linear.

Clearly $\Ftwo$ is a subclass of $\Fone$ and all the results of the previous section for functions in $\Fone$ are also valid for functions in $\Ftwo$.

\begin{lemma}\label{lemma:identification_sum_functions}
Given the non-zero analytic functions $f_i:\R\to\R$ and $g_i,\tilde g_i:\R^m\to \R$ satisfying  $f_i(0)=g_i(0)=\tilde{g}_i(0)=0$ for $i=1,\ldots,n$. Let us assume that none of the functions $f_i$ is linear. If for all $x\in\R$, $y\in\R^m$, the functions $f_i$, $g_i$ and $\tilde g_i$ satisfy:
$$
\sum_{i=1}^nf_i(x_i+g_i(y_1,\ldots,y_m))=\sum_{i=1}^nf_i(x_i+\tilde{g}_i(y_1,\ldots,y_m)),
$$
then $g_i=\tilde g_i$ for all $i=1,\ldots,n$.
\end{lemma}
\begin{proof}
The proof is left to Appendix~\ref{app:3}.
\end{proof}

Notice that when $n=1$, Lemma~\ref{lemma:identification_sum_functions} is also covered by Corollary~\ref{corr:g_egal_g}.

\begin{proposition}\label{prop:path_2_degree}
For the functions in $\Ftwo$, in a directed acyclic graph, the measurement of a node provides the identification of all the nonlinear functions of any path that finishes in the measured node.
\end{proposition}
\begin{proof}
Let us assume an arbitrary directed acyclic graph. The measurement of a node $i$ provides an output of the type:
\begin{align*}
    y_i^k&=u_i^{k-1}+F_i^k\\
    &=u_i^{k-1}+\sum_{j=1}^mf_{i,j}(u_j^{k-2}+F_j^{k-1}),
\end{align*}
where $m$ is the number of in-neighbors of $i$. Let us assume that there is a set $\{ \tilde f\}\neq \{ f\}$ such that $F_i^k=\tilde F_i^k$, which implies:
$$
\sum_{j=1}^mf_{i,j}(u_j^{k-2}+F_j^{k-1})=\sum_{j=1}^m\tilde f_{i,j}(u_j^{k-2}+\tilde F_j^{k-1}),
$$
By applying Lemma~\ref{lemma:sinks}
we have $f_{i,j}=\tilde f_{i,j}$ for all $j=1,\ldots,m$ and:
$$
\sum_{j=1}^mf_{i,j}(u_j^{k-2}+F_j^{k-1})=\sum_{j=1}^mf_{i,j}(u_j^{k-2}+\tilde F_j^{k-1}),
$$
and by using Lemma~\ref{lemma:identification_sum_functions} we guarantee:
$$
F_j^{k-1}=\tilde F_j^{k-1} \quad \text{for all }j=1,\ldots,m.
$$
Notice that the identification of each $F_j^{k-1}$ is equivalent to having measured the node $j$ and can be treated in a similar way to the node $i$, independently of other paths corresponding to the other in-neighbors of $i$.  By following a similar approach, we can guarantee that $\{ f \}=\{ \tilde f \}$ for every path that ends in the node $i$.
\end{proof}

\begin{theorem}[Directed acyclic graph]\label{thm:DAG}
For identifiability of a directed acyclic graph in the class $\Ftwo$, it is necessary and sufficient to measure all the sinks.
\end{theorem}
\begin{proof}
From Proposition~\ref{prop:sinks_sources}, it is necessary to measure all the sinks. In a directed acyclic graph, we can always find a path from any node $i$ to some sink \cite{bang2008digraphs}. Therefore, according to Proposition~\ref{prop:path_2_degree}, it is sufficient to measure the sinks to identify all the paths in a directed acyclic graph.
\end{proof}

Unlike the linear case, where the measurement of the sinks is not enough to guarantee identifiability of directed acyclic graphs \cite{hendrickx2019identifiability}, Theorem~\ref{thm:DAG} provides weaker conditions for the identifiability in the nonlinear case when linear functions are excluded. 

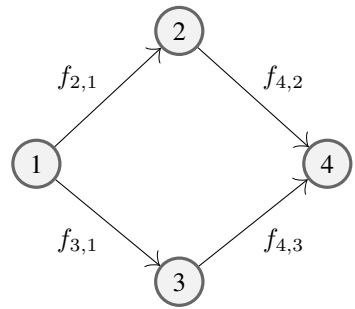
\begin{figure}
    \centering

\vspace{2mm}
    
    \begin{tikzpicture}
    [
roundnodes/.style={circle, draw=black!60, fill=black!5, very thick, minimum size=1mm},roundnode/.style={circle, draw=white!60, fill=white!5, very thick, minimum size=1mm}
]
\node[roundnodes](node1){1};
\node[roundnodes](node2)[above=of node1,yshift=1mm,xshift=1.9cm]{2};
\node[roundnodes](node4)[right=3.2cm of node1]{4};
\node[roundnodes](node3)[below=of node1,yshift=1mm,xshift=1.9cm]{3};

\draw[-{Classical TikZ Rightarrow[length=1.5mm]}] (node1) to node [left,swap,yshift=2.5mm] {$f_{2,1}$} (node2);
\draw[-{Classical TikZ Rightarrow[length=1.5mm]}] (node2) to node [right,swap,yshift=2.5mm] {$f_{4,2}$} (node4);
\draw[-{Classical TikZ Rightarrow[length=1.5mm]}] (node1) to node [left,swap,yshift=-2.5mm] {$f_{3,1}$} (node3);
\draw[-{Classical TikZ Rightarrow[length=1.5mm]}] (node3) to node [right,swap,yshift=-2.5mm] {$f_{4,3}$} (node4);
\end{tikzpicture}
    \caption{Nonlinear network that can be identified by only measuring the sink. In the linear case, the measurement of the sink is not enough to identify this network.}
    \label{fig:bridge_graph}
\end{figure}

\begin{example}[Directed acyclic graph] 
Let us consider the graph in Fig.~\ref{fig:bridge_graph}. In the linear case, this network cannot be identified by only measuring the sink since the functions $f_{2,1}$ and $f_{3,1}$ cannot be distinguished. However, in the nonlinear case, the measurement of the sink is enough to identify all the network.

\end{example}

\section{Conclusions and future work}

We have derived identifiability conditions for a network characterized by nonlinear interactions through a static model. We showed that in a path graph it is necessary to measure all the nodes, except by the source, when the nonlinear functions have a static component. Then, by restricting the identifiability problem to a specific class of functions, we showed that the measurement of the sinks is necessary and sufficient to identify all the edges in paths and trees. Finally, by considering a smaller class of functions, we showed that the measurement of the sinks is necessary and sufficient for the identifiability of directed acyclic graphs. This simple model of nonlinear interactions allowed us to highlight fundamental differences with respect to the 
linear case.

For future work, it would be interesting to extend the results to the case of general digraphs with cycles where the function $F_i^k$ in \eqref{eq:function_Fi} depends on an infinite number of inputs. Also, it would be important to consider dynamical models that include past inputs. In this case, the Volterra series seems to be the more adequate model since it only depends on past inputs, and many of our results could still hold.

\bibliographystyle{IEEEtran}
\bibliography{CDC}

\appendix

\subsection{Proof of Lemma~\ref{lemma:periodic}}\label{app:1}

Let us choose an arbitrary point $\hat x$. Since $f$ is periodic for $p\in[p_0-\epsilon,p_0+\epsilon]$, we have that for all $y\in[\hat x-\epsilon +p_0,\hat x+\epsilon +p_0]$:
   $$
   f(y)=f(\hat x),
   $$
   which implies that the function $f$ is constant and its derivative $f'$ is zero in $[\hat x-\epsilon +p_0,\hat x+\epsilon +p_0]$. Since $f$ is analytic, its derivative $f'$ is also analytic and due to the Principle of isolated zeros for 1-dimensional real analytic functions \cite{krantz2002primer}, the derivative $f'$ is zero for all $x$, which implies that $f$ is constant for \linebreak all $x$.

\subsection{Proof of Lemma~\ref{lemma:identification_one_function}}\label{app:2}

Let us assume that there exists a point $\hat y\in \R^m$ such that $g(\hat y)=a$ and $\tilde g(\hat y)=b$ with $a\neq b$. Then, we would have:
   $$
   f(x+a)=f(x+b) , \text{ for all } x\in\R,
   $$
   which is equivalent to 
   $$
   f(z)=f(z+b-a), \text{ for all } z\in\R,
   $$
   implying that $f$ is periodic with period $b-a$. Since $g(0)=\tilde g(0)=0$, and $g$ and $\tilde{g}$ are continuous, the function $\tilde g-g$ is also continuous and all the values between 0 and $b-a$ belong to its range. 
   Thus, $f$ should be periodic in the interval $[0,b-a]$. But by virtue of Lemma~\ref{lemma:periodic}, the function $f$ should be constant. If $f$ is not constant, we have a contradiction which implies \linebreak that $g=\tilde{g}$. 

\subsection{Proof of Lemma~\ref{lemma:identification_sum_functions}}\label{app:3}

Let us take the derivative with respect to only one variable $x_j$ where $j=1,\ldots,n$. Then, we have:
$$
f'_j(x_j+g_j(y))=f'_j(x_j+\tilde g_j(y)), \text{ for all } x_j\in\R,\text{ } y\in\R^m.
$$
Since the function $f_j$ is analytic, its derivative $f'_j$ is also analytic, and by Lemma~\ref{lemma:identification_one_function}, we can have two cases: either $f'_j$ is constant or $g_j=\tilde g_j$. If $f'_j$ is constant, then $f_j$ is linear or constant (i.e., $f_j=0$), which is a contradiction. 
Therefore $g_j=\tilde g_j$. Following the same procedure for the other variables $x_j$, we complete the proof.

\end{document}